\documentclass[11pt]{amsart}

\usepackage{amsmath, amssymb, amsthm, amsfonts}
\usepackage{mathrsfs}
\usepackage{enumerate}
\usepackage{hyperref}
\usepackage{geometry}
\theoremstyle{plain}
\newtheorem{theorem}{Theorem}[section]
\newtheorem{proposition}[theorem]{Proposition}
\newtheorem{lemma}[theorem]{Lemma}

\theoremstyle{definition}

\theoremstyle{remark}
\newtheorem{remark}[theorem]{Remark}

\DeclareMathOperator{\SE}{SE}
\DeclareMathOperator{\SO}{SO}

\DeclareMathOperator{\Range}{Range}
\DeclareMathOperator{\Ker}{Ker}

\numberwithin{equation}{section}

\title{Cohomological Equation for Robotic Screw Motions on the Lie Group \mbox{$\SE(3)$}}

\author{Amanze C. Egere}
\address{Department of Robotics/Interdisciplinary Engineering, Purdue University}
\email{aegere@purdue.edu}

\date{\today}


\begin{document}

\begin{abstract}
We study the cohomological equation associated with screw motions on the Euclidean
motion group $\SE(3)$.
Working on a (smooth) manifold  $M := \mathbb{T}^3 \times \SO(3)$ (or phase space), we combine Fourier
analysis in the translational variables with Peter-Weyl theory on $\SO(3)$ to
reduce the equation to a family of finite-dimensional linear cocycle (or transport system) along frequency orbits induced by the rotational component.
In the case of finite-order rotations, solvability is governed by explicit
finite-dimensional linear obstructions encoded by monodromy operators.
An explicit screw motion along the $z$-axis illustrates the resulting resonance
conditions.
Since rigid motions on $\SE(3)$ arise naturally as configuration spaces in robotic
kinematics, the results provide a mathematically precise description of
obstruction phenomena relevant to robotic rigid--body motions.
\end{abstract}

\maketitle

\section{Introduction}

Cohomological equations play a central role in smooth dynamics and rigidity theory,
encoding obstructions to conjugacy, regularity, and deviation phenomena.
Given a diffeomorphism $\gamma$ acting on a manifold $M$, the equation
\begin{equation}
f \circ \gamma - f = g
\label{eq:cohom}
\end{equation}
carries fundamental information about invariant distributions and the behavior of ergodic averages. The systematic study of such equations dates back to the foundational work of
Gottschalk and Hedlund~\cite{GottschalkHedlund1955} in topological dynamics.
In the smooth setting, connections with rigidity and global hypoellipticity
were uncovered by Greenfield and Wallach~\cite{GreenfieldWallach1973}, and later
developed by Forni~\cite{Forni1997}, who identified invariant distributions as
the primary obstructions to solvability.

For abelian groups and nilmanifolds, cohomological equations allow for strong reductions based on Fourier analysis and representation theory. Specifically, when considering translations on tori and nilrotations on the Heisenberg group, equation~\eqref{eq:cohom} simplifies to concrete algebraic conditions on the Fourier coefficients, which expose the exact resonance obstructions and permit the derivation of sharp regularity results.
The purpose of this paper is to establish an analogous reduction for screw motions on the Euclidean motion group  $\SE(3)$.
Screw motions combine rotation about a given axis with translation along the same axis and constitute a fundamental class of rigid motions in geometry and mechanics.
Rigid motions on $\SE(3)$ arise naturally as configuration spaces in robotic
kinematics. This provides additional motivation for the present analysis.

Although the harmonic analysis in $\SE(3)$ is classical, the corresponding cohomological
equations for rigid motions involving rotations have received relatively
little attention. Unlike the nilpotent case, the presence of rotations introduces two new features:
the transport of Fourier frequencies along rotational orbits and the appearance
of matrix-valued cocycles arising from the representation theory of $\SO(3)$.
In this article, we show that, despite these complications, the cohomological equation retains a
tractable algebraic structure and, in the case of finite-order rotations, it reduces
to finite-dimensional linear systems with explicit solvability criteria.
An explicit axial screw motion is analyzed in Section~\ref{sec:zaxis}, where the
obstruction mechanism reduces to a transparent resonance condition depending
only on the translational pitch and axial Fourier frequencies.

\section{Background and Related Work}

The study of cohomological equations has a long history in dynamical systems.
For translations on tori, solvability and regularity are governed by Diophantine
conditions on the translation vector, reflecting the interaction between
arithmetic properties and small--divisor phenomena.

On homogeneous spaces, and in particular on nilmanifolds, representation--theoretic
methods provide a powerful framework for analyzing cohomological equations.
In this setting, solvability can often be reduced to algebraic conditions on
irreducible components of the regular representation.
The obstruction theory arising from invariant distributions was developed
systematically in the study of unipotent flows and nilflows, where it plays
a central role in rigidity phenomena and deviation estimates.

A comprehensive representation--theoretic treatment of cohomological equations
on nilmanifolds, including the Heisenberg group, was given by Flaminio and
Forni~\cite{FlaminioForni2007}.
Their work shows that solvability reduces to finite--dimensional obstruction
spaces associated with irreducible unitary representations, extending earlier
analyses of invariant distributions for unipotent flows~\cite{FlaminioForni2003}.
Related techniques also appear in the study of cocycle rigidity and smooth
conjugacy problems over group actions~\cite{KatokSpatzier1996}.

More recently, Egwe~\cite{Egwe2025} studied cohomological equations on the
discrete Heisenberg group, providing explicit solvability conditions in a
noncommutative setting that shares structural features with the systems
considered here.

The Euclidean motion group $\SE(3)$ occupies a different position in the Lie
group hierarchy, being neither abelian nor nilpotent.
While harmonic analysis on $\SE(3)$ is classical, the study of cohomological
equations for rigid motions involving rotations has received comparatively
little attention.
Screw motions therefore provide a natural and geometrically meaningful test
case in which explicit analysis remains possible.

General background on flows on homogeneous spaces may be found in
Auslander, Green, and Hahn~\cite{AuslanderGreenHahn1963}.

\section{The Group $\SE(3)$ and a Compact Phase Space}

The Euclidean motion group in three dimensions is the semidirect product
\[
\SE(3) = \mathbb{R}^3 \rtimes \SO(3),
\]
with group law
\[
(x,R)\cdot(y,S) = (x + Ry, RS),
\qquad (x,R)^{-1} = (-R^{T}x, R^{T}).
\]

To place the dynamics in a compact setting compatible with Fourier and representation-theoretic methods, we work on
\[
M := \mathbb{T}^3 \times \SO(3),
\qquad \mathbb{T}^3 = \mathbb{R}^3 / \mathbb{Z}^3,
\]
endowed with Haar probability measure
\[
\mu = m_{\mathbb{T}^3} \otimes m_{\SO(3)}.
\]

Throughout, we assume
\[
R_0 \in \SO(3) \cap GL(3,\mathbb{Z}),
\]
so that the map $x \mapsto R_0 x$ descends to an automorphism of the torus
$\mathbb{T}^3$.
This condition implies that $R_0$ has finite order and acts by finite
permutations on Fourier frequencies.
As a consequence, the induced action decomposes frequency space into finite
orbits, allowing the cohomological equation to be reduced to finite--dimensional
linear systems.
For general harmonic analysis and probabilistic methods on $\SE(3)$, see
Chirikjian~\cite{Chirikjian2011}.
.

\section{The Discrete Dynamical System}

Let $a = (t,R_0) \in \SE(3)$ be a screw motion.
Left translation by $a$ defines a diffeomorphism
\[
\gamma : M \to M,
\qquad
\gamma(x,R) = (t + R_0 x,\; R_0 R).
\]
Since translations on $\mathbb{T}^3$ and left multiplication on $\SO(3)$ preserve Haar measure, $\gamma$ preserves $\mu$.

\section{The Cohomological Equation}

We study the cohomological equation
\begin{equation}
f \circ \gamma - f = g,
\label{eq:cohomological}
\end{equation}
for functions $f,g \in L^2(M)$.

\begin{lemma}
A necessary condition for solvability of \eqref{eq:cohomological} is
\[
\int_M g \, d\mu = 0.
\]
\end{lemma}

\begin{proof} Assume that there exists a measurable solution $f$ to \eqref{eq:cohomological}
such that $f,g\in L^{1}(M,\mu)$ (in particular, this holds if $f,g\in L^{2}(M,\mu)$
since $\mu(M)=1$).
Integrating \eqref{eq:cohomological} over $M$ and using $\gamma$--invariance of $\mu$ yields
\[
\int_M (f \circ \gamma - f)\, d\mu = 0.
\]
\end{proof}

\section{Fourier Analysis on $\mathbb{T}^3$ and Peter-Weyl Theory on $\SO(3)$}
We rely on standard harmonic analysis on compact groups and Lie groups;
background on Fourier analysis on $\mathbb{T}^3$ and representation theory
of compact Lie groups may be found in
Knapp~\cite{Knapp2002} and Folland~\cite{Folland1995}. For $k \in \mathbb{Z}^3$, define
\[
e_k(x) := e^{2\pi i\, k \cdot x},
\]
which form an orthonormal basis of $L^2(\mathbb{T}^3)$.
A direct computation shows that
\begin{equation}
e_k(t + R_0 x)
=
e^{2\pi i\, k \cdot t}\, e_{R_0^{T} k}(x).
\label{eq:torus-transport}
\end{equation}

By the Peter-Weyl theorem, the matrix coefficients
\[
\{D^\ell_{mn}(R) : \ell \in \mathbb{N}_0,\; -\ell \le m,n \le \ell\}
\]
form an orthogonal basis of $L^2(\SO(3))$, where $D^\ell$ denotes the irreducible unitary representation of degree $2\ell+1$.
These satisfy
\begin{equation}
D^\ell(R_0 R) = D^\ell(R_0)\, D^\ell(R).
\label{eq:pw-transport}
\end{equation}

Every $f \in L^2(M)$ admits an expansion
\[
f(x,R)
=
\sum_{k \in \mathbb{Z}^3}
\sum_{\ell=0}^{\infty}
\sum_{m,n=-\ell}^{\ell}
f_{k,\ell,mn}\, e_k(x)\, D^\ell_{mn}(R),
\]
and similarly for $g$. See also Folland and Stein~\cite{FollandStein1982} for Hardy space methods on
homogeneous groups.

\section{The Reduced Linear System}

Fix $\ell \ge 0$ and $n \in \{-\ell,\dots,\ell\}$.
Define vector-valued Fourier coefficients
\[
F_{\ell,n}(k)
:= \big(f_{k,\ell,mn}\big)_{m=-\ell}^{\ell}
\in \mathbb{C}^{2\ell+1},
\qquad
G_{\ell,n}(k)
:= \big(g_{k,\ell,mn}\big)_{m=-\ell}^{\ell}.
\]
Let
\[
U_\ell := D^\ell(R_0).
\]

\begin{proposition}[Reduced transport equation]\label{prop:reduced-transport}

The cohomological equation \eqref{eq:cohomological} is equivalent to the system
\begin{equation}
e^{2\pi i (R_0 k)\cdot t}\, U_\ell\, F_{\ell,n}(R_0 k)
-
F_{\ell,n}(k)
=
G_{\ell,n}(k),
\qquad k \in \mathbb{Z}^3,
\label{eq:reduced-system}
\end{equation}
for every $\ell\ge 0$ and $n\in\{-\ell,\dots,\ell\}$, where $U_\ell=D^\ell(R_0)$.
\end{proposition}

\begin{proof}
Write the Fourier--Peter--Weyl expansions
\[
f(x,R)=\sum_{k\in\mathbb{Z}^3}\sum_{\ell\ge 0}\sum_{m,n=-\ell}^{\ell}
f_{k,\ell,mn}\, e_k(x)\, D^\ell_{mn}(R),
\]
\[
g(x,R)=\sum_{k\in\mathbb{Z}^3}\sum_{\ell\ge 0}\sum_{m,n=-\ell}^{\ell}
g_{k,\ell,mn}\, e_k(x)\, D^\ell_{mn}(R),
\]
where $e_k(x)=e^{2\pi i k\cdot x}$ and $D^\ell_{mn}$ are matrix coefficients of the irreducible representation $D^\ell$.

Recall that $\gamma(x,R)=(t+R_0x,\;R_0R)$, so
\[
(f\circ\gamma)(x,R)=f(t+R_0x,\;R_0R).
\]
Substituting the expansion of $f$ gives
\begin{align*}
(f\circ\gamma)(x,R)
&=
\sum_{k,\ell,m,n} f_{k,\ell,mn}\,
e_k(t+R_0x)\,
D^\ell_{mn}(R_0R).
\end{align*}

We now use the transport identities on $\mathbb{T}^3$ and $\SO(3)$.
First, by \eqref{eq:torus-transport},
\[
e_k(t+R_0x)=e^{2\pi i k\cdot t}\, e_{R_0^T k}(x).
\]
Second, by \eqref{eq:pw-transport} (i.e.\ $D^\ell(R_0R)=D^\ell(R_0)D^\ell(R)$), we have for matrix coefficients
\[
D^\ell_{mn}(R_0R)=\sum_{m'=-\ell}^{\ell} D^\ell_{m m'}(R_0)\, D^\ell_{m'n}(R).
\]
Therefore,
\begin{align*}
(f\circ\gamma)(x,R)
&=\sum_{k,\ell,m,n} f_{k,\ell,mn}\,
\big(e^{2\pi i k\cdot t}\, e_{R_0^T k}(x)\big)\,
\sum_{m'=-\ell}^{\ell} D^\ell_{m m'}(R_0)\, D^\ell_{m'n}(R)\\
&=\sum_{k,\ell,n}\; e^{2\pi i k\cdot t}\, e_{R_0^T k}(x)\,
\sum_{m'=-\ell}^{\ell}
\left(\sum_{m=-\ell}^{\ell} D^\ell_{m m'}(R_0)\, f_{k,\ell,mn}\right)
D^\ell_{m'n}(R).
\end{align*}

We now re-index the $\mathbb{T}^3$ frequency to express everything in the basis $\{e_k(x)\}$.
Let $k' = R_0^T k$, i.e.\ $k = R_0 k'$ (since $R_0$ is orthogonal).
Then $e_{R_0^T k}(x)=e_{k'}(x)$ and
\[
k\cdot t = (R_0 k')\cdot t.
\]
Hence
\begin{align*}
(f\circ\gamma)(x,R)
&=\sum_{k'\in\mathbb{Z}^3}\sum_{\ell,n}
e^{2\pi i (R_0 k')\cdot t}\, e_{k'}(x)\,
\sum_{m'=-\ell}^{\ell}
\left(\sum_{m=-\ell}^{\ell} D^\ell_{m m'}(R_0)\, f_{R_0 k',\ell,mn}\right)
D^\ell_{m'n}(R).
\end{align*}

Now compare the cohomological equation $f\circ\gamma - f = g$ in the orthogonal basis
$\{e_k(x)D^\ell_{m'n}(R)\}$.
Fix $k\in\mathbb{Z}^3$, $\ell\ge 0$, and $n\in\{-\ell,\dots,\ell\}$.
For each $m'\in\{-\ell,\dots,\ell\}$, the coefficient of $e_k(x)D^\ell_{m'n}(R)$ on the left-hand side equals
\[
e^{2\pi i (R_0 k)\cdot t}\,
\sum_{m=-\ell}^{\ell} D^\ell_{m m'}(R_0)\, f_{R_0 k,\ell,mn}
\;-\;
f_{k,\ell,m'n},
\]
and must equal the corresponding coefficient $g_{k,\ell,m'n}$ of $g$.

Introduce the vector notation
\[
F_{\ell,n}(k) := (f_{k,\ell,mn})_{m=-\ell}^{\ell},
\qquad
G_{\ell,n}(k) := (g_{k,\ell,mn})_{m=-\ell}^{\ell},
\qquad
U_\ell := D^\ell(R_0).
\]
Then the previous componentwise identity is exactly the vector equation
\[
e^{2\pi i (R_0 k)\cdot t}\, U_\ell\, F_{\ell,n}(R_0 k)
-
F_{\ell,n}(k)
=
G_{\ell,n}(k),
\]
which is \eqref{eq:reduced-system}.
Conversely, \eqref{eq:reduced-system} implies equality of all Fourier-Peter-Weyl coefficients of $f\circ\gamma-f$ and $g$, hence \eqref{eq:cohomological}.
\end{proof}

\section{Orbit Reduction and Obstructions}\label{sec:obstructions}

Since $R_0 \in GL(3,\mathbb{Z}) \cap \SO(3)$ has finite order, there exists
$p \ge 1$ such that $R_0^p = I$.
For each $k \in \mathbb{Z}^3$, define the finite orbit
\[
k_j := R_0^{\,j} k,
\qquad j = 0,\dots,p-1.
\]

Let
\[
F_j := F_{\ell,n}(k_j),
\qquad
G_j := G_{\ell,n}(k_j),
\qquad
\alpha_j := e^{2\pi i\, k_{j+1} \cdot t}.
\]
Then \eqref{eq:reduced-system} becomes the cyclic system
\[
\alpha_j\, U_\ell\, F_{j+1} - F_j = G_j,
\qquad j = 0,\dots,p-1,
\]
with indices taken modulo $p$.

Iterating around the orbit yields the closure equation
\[
(A_\ell(k) - I)\, F_0 = B_{\ell,n}(k),
\]
where
\[
A_\ell(k)
=
\left( \prod_{j=0}^{p-1} \alpha_j \right) U_\ell^{\,p},
\qquad
B_{\ell,n}(k)
=
\sum_{r=0}^{p-1}
\left( \prod_{j=0}^{r-1} \alpha_j U_\ell \right) G_r.
\]
The finite--orbit reduction obtained below mirrors the obstruction theory
that appears in the study of cohomological equations for nilflows and
partially hyperbolic systems, where solvability is governed by finite--
dimensional invariant distributions; see, for example,
Forni~\cite{Forni1997} and Giulietti--Liverani~\cite{GiuliettiLiverani2013}.

\begin{theorem}[Finite-orbit solvability and obstructions]
\label{thm:finite-orbit}
The cohomological equation \eqref{eq:cohomological} admits a solution
$f \in L^2(M)$ if and only if
\[
\int_M g\, d\mu = 0
\quad\text{and}\quad
B_{\ell,n}(k) \in \Range(A_\ell(k) - I)
\]
for all $k,\ell,n$.
Equivalently,
\[
\langle B_{\ell,n}(k), v \rangle = 0
\quad
\text{for all }
v \in \Ker(A_\ell(k)^* - I).
\]
\end{theorem}

\begin{proof}
We have already shown that $\int_M g\,d\mu=0$ is necessary for solvability of
\eqref{eq:cohomological}. We therefore focus on the Fourier-Peter-Weyl
reduction and the orbit-by-orbit criterion.

Fix $\ell\ge 0$ and $n\in\{-\ell,\dots,\ell\}$. By Proposition
\ref{eq:reduced-system} (Reduced transport equation), solvability of
\eqref{eq:cohomological} is equivalent to solvability of
\eqref{eq:reduced-system} for every $(\ell,n)$, i.e.
\begin{equation}
\alpha(k)\,U_\ell\,F_{\ell,n}(R_0k) - F_{\ell,n}(k) = G_{\ell,n}(k),
\qquad k\in\mathbb{Z}^3,
\label{eq:transport-rewrite}
\end{equation}
where $\alpha(k):=e^{2\pi i (R_0 k)\cdot t}$ and $U_\ell=D^\ell(R_0)$.

\smallskip
\noindent\textbf{Step 1: Finite orbit reduction and the closure equation.}
Since $R_0$ has finite order, choose $p\ge 1$ such that $R_0^p=I$.
Fix a frequency $k\in\mathbb{Z}^3$ and denote its orbit by
\[
k_j := R_0^{\,j}k,\qquad j=0,\dots,p-1.
\]
For brevity write
\[
F_j := F_{\ell,n}(k_j),\qquad G_j := G_{\ell,n}(k_j),
\qquad \alpha_j := e^{2\pi i\,k_{j+1}\cdot t},
\]
where indices are taken modulo $p$.
Then \eqref{eq:transport-rewrite} restricted to this orbit becomes the cyclic
system
\begin{equation}
\alpha_j\,U_\ell\,F_{j+1} - F_j = G_j,
\qquad j=0,\dots,p-1.
\label{eq:cyclic-system}
\end{equation}

Rearrange \eqref{eq:cyclic-system} as
\begin{equation}
F_j = \alpha_j\,U_\ell\,F_{j+1} - G_j.
\label{eq:backward}
\end{equation}
Iterating \eqref{eq:backward} expresses $F_0$ in terms of $F_p$ and the forcing
$\{G_j\}$.
Indeed,
\begin{align*}
F_0
&= \alpha_0 U_\ell F_1 - G_0 \\
&= \alpha_0 U_\ell(\alpha_1 U_\ell F_2 - G_1) - G_0 \\
&= \alpha_0\alpha_1 U_\ell^2 F_2 - \alpha_0U_\ell G_1 - G_0\\
&\ \ \vdots \\
&= \Big(\prod_{j=0}^{p-1}\alpha_j\Big) U_\ell^{\,p} F_p
\;-\;
\sum_{r=0}^{p-1}\Big(\prod_{j=0}^{r-1}\alpha_j U_\ell\Big)G_r.
\end{align*}
Since $k_p = R_0^{\,p}k = k_0$, we have $F_p=F_0$.
Define the monodromy operator
\[
A_\ell(k) := \Big(\prod_{j=0}^{p-1}\alpha_j\Big) U_\ell^{\,p},
\]
and the twisted orbit-sum
\[
B_{\ell,n}(k) :=
\sum_{r=0}^{p-1}\Big(\prod_{j=0}^{r-1}\alpha_j U_\ell\Big)G_r,
\]
where the empty product (for $r=0$) is interpreted as the identity.
The previous identity becomes
\[
F_0 = A_\ell(k)F_0 - B_{\ell,n}(k),
\]
equivalently the closure equation
\begin{equation}
(A_\ell(k) - I)F_0 = B_{\ell,n}(k).
\label{eq:closure-eq}
\end{equation}

\smallskip
\noindent\textbf{Step 2: Solvability on an orbit is equivalent to solving the closure equation.}
We claim that the cyclic system \eqref{eq:cyclic-system} has a solution
$(F_0,\dots,F_{p-1})$ if and only if there exists $F_0\in\mathbb{C}^{2\ell+1}$
satisfying \eqref{eq:closure-eq}.

First, if $(F_0,\dots,F_{p-1})$ solves \eqref{eq:cyclic-system}, then the above
iteration is valid and necessarily yields \eqref{eq:closure-eq}.

Conversely, suppose $F_0$ satisfies \eqref{eq:closure-eq}.
We reconstruct the entire orbit uniquely from $F_0$ by solving forward:
from \eqref{eq:cyclic-system} we have
\[
\alpha_j U_\ell F_{j+1} = F_j + G_j,
\qquad\text{hence}\qquad
F_{j+1} = \alpha_j^{-1}U_\ell^{-1}(F_j+G_j).
\]
Since $\alpha_j\in\mathbb{S}^1$ and $U_\ell$ is unitary, $\alpha_j^{-1}U_\ell^{-1}$
is invertible, so $F_1,\dots,F_{p-1}$ are determined recursively.
The only remaining requirement is consistency after one full cycle, i.e.\ that
the constructed $F_p$ equals $F_0$.
But this consistency is exactly the closure equation \eqref{eq:closure-eq}.
Therefore \eqref{eq:cyclic-system} is solvable on the orbit of $k$ if and only if
\eqref{eq:closure-eq} is solvable.

Thus, solvability of the reduced transport system \eqref{eq:transport-rewrite}
for all $k$ is equivalent to solvability of \eqref{eq:closure-eq} for all orbits
and all $(\ell,n)$.

\smallskip
\noindent\textbf{Step 3: Linear algebra and the obstruction condition.}
Equation \eqref{eq:closure-eq} is a linear equation in the finite-dimensional
vector space $\mathbb{C}^{2\ell+1}$.
It has a solution if and only if
\[
B_{\ell,n}(k) \in \Range(A_\ell(k)-I),
\]
which proves the first equivalence in the theorem.

Finally, we show the standard orthogonality characterization.
For any linear operator $T$ on a finite-dimensional inner-product space,
\[
\Range(T) = (\Ker(T^*))^\perp.
\]
Apply this to $T=A_\ell(k)-I$. Then $T^* = A_\ell(k)^*-I$, and we obtain
\[
B_{\ell,n}(k) \in \Range(A_\ell(k)-I)
\quad\Longleftrightarrow\quad
\langle B_{\ell,n}(k), v \rangle = 0\ \ \text{for all } v\in\Ker(A_\ell(k)^*-I).
\]
This yields the equivalent obstruction formulation and completes the proof.
\end{proof}

\begin{remark}
The obstruction mechanism above arises from the interaction between
frequency transport and finite--dimensional representation mixing,
and reduces to a scalar resonance condition in the case of axial
screw motions.
\end{remark}

\section{An Explicit Screw Motion Along the $z$--Axis}
\label{sec:zaxis}

We conclude with an explicit example illustrating the obstruction mechanism
in a particularly transparent setting.
Let $R_0 \in \SO(3)$ be rotation by angle $\theta = 2\pi q/p$ about the $z$--axis,
where $p,q \in \mathbb{Z}$ are coprime, and let the translational component be
\[
t = (0,0,h) \in \mathbb{T}^3.
\]
Then $R_0^p = I$, and the induced action on frequency space decomposes into
finite orbits of length dividing $p$.

For any $k=(k_x,k_y,k_z)\in\mathbb{Z}^3$, the orbit $\{k_j=R_0^j k\}_{j=0}^{p-1}$
preserves the $z$--component, so that
\[
k_j \cdot t = k_z h
\qquad \text{for all } j.
\]
Consequently, the scalar phase factors in the monodromy operator satisfy
\[
\prod_{j=0}^{p-1} e^{2\pi i\, k_{j+1}\cdot t}
=
e^{2\pi i p k_z h}.
\]

Since $U_\ell = D^\ell(R_0)$ is unitary and $U_\ell^p = I$ for a rotation by
$2\pi q/p$, the monodromy operator defined in Section~\ref{sec:obstructions}
reduces in this case to the scalar operator
\[
A_\ell(k) = e^{2\pi i p k_z h}\, I
\]
on $\mathbb{C}^{2\ell+1}$.
The closure equation $(A_\ell(k)-I)F_0 = B_{\ell,n}(k)$ therefore decouples
componentwise and admits a solution if and only if
\[
e^{2\pi i p k_z h} \neq 1
\quad\text{or}\quad
B_{\ell,n}(k)=0.
\]

Equivalently, resonance occurs precisely when
\[
p k_z h \in \mathbb{Z},
\]
in which case the obstruction space coincides with the full representation
space $\mathbb{C}^{2\ell+1}$.
Outside this resonance set, the operator $A_\ell(k)-I$ is invertible and the
cohomological equation admits a unique solution along each frequency orbit.

This example shows that, for axial screw motions, solvability is governed
solely by the interaction between the translational pitch $h$ and the axial
Fourier frequency $k_z$, with no contribution from transverse frequencies.
The resulting resonance condition closely parallels the scalar small-divisor
phenomena appearing in toral translations and highlights the finite-dimensional
nature of obstructions in the present compact setting.

\section{Conclusion}

We have obtained a complete Fourier--representation-theoretic reduction of the cohomological equation for screw motions on $\SE(3)$.
In the compact, finite-order rotation setting, solvability reduces to explicit finite-dimensional obstructions analogous to those arising in nilmanifold dynamics.
The explicit screw motion along the $z$--axis demonstrates that resonances are governed solely by the translational pitch and axial frequencies. Extensions to irrational rotations and noncompact settings lead to
infinite-frequency cocycles and small-divisor phenomena and will be
addressed in future work.

\bibliographystyle{amsplain}
\bibliography{references}

\end{document}